\documentclass[a4paper,twoside,10pt]{article}
\usepackage{authblk}
\usepackage{fancyhdr}
\usepackage{amsmath}
\usepackage{amsthm}
\usepackage{amsfonts}
\usepackage[T1]{fontenc}

\theoremstyle{plain}
\newtheorem{theorem}{Theorem}[section]
\newtheorem{corollary}[theorem]{Corollary}
\newtheorem{lemma}[theorem]{Lemma}
\newtheorem{proposition}[theorem]{Proposition}

\theoremstyle{definition}
\newtheorem{definition}[theorem]{Definition}

\newtheorem{example}[theorem]{Example}
\newtheorem{exercise}[theorem]{Exercise}

\newenvironment{result}[1]%
{\medskip\noindent\textbf{#1.}~\it}{\medskip}

\newcommand{\N}{\mathbb N}
\newcommand{\Z}{\mathbb Z}

\newcommand{\R}{\mathbb R}
\newcommand{\C}{\mathbb C}

\newcommand{\sign}{\mathop\mathrm{sign}\nolimits}

\newcommand{\dist}{\mathop\mathrm{dist}\nolimits}

\newcommand{\GL}{\mathrm{GL}}

\newcommand{\cl}[1]{\overline{#1}}
\renewcommand{\L}{\mathrm{L}}
\renewcommand{\i}{\mathrm{i}}

\author[$1$]{Pierluigi Benevieri}
\author[$2$]{Massimo Furi}
\author[$3$]{Maria Patrizia Pera}
\author[$4$]{Marco Spadini}

\affil[$1$]{\small Instituto de Matem\'atica e Estat\'istica, Universidade de S\~ao Paulo, Brasil}
\affil[$2$,$3$,$4$]{\small Dipartimento di Matematica e Informatica ``Ulisse Dini'', Universit\`a di Firenze, Italy}
\title{An introduction to topological degree in Euclidean spaces}
\date{}

\numberwithin{equation}{section}

\begin{document}
    
\maketitle


\section{Introduction}

This paper aims to provide a careful and self-contained introduction to the theory
of topological degree in Euclidean spaces. It is intended for people mostly interested
in analysis and, in general, a heavy background in algebraic or differential topology
is not required. 

Roughly speaking, our construction of the topological degree can be summarized in a 
few steps. We first define a notion of degree for the special case of \emph{regular 
triples} that is for triples $(f,U,y)$ where $f$ is an $\R^k$-valued smooth
function defined (at least) on the closure $\cl{U}$ of the open set $U\subseteq\R^k$ 
and proper on $\cl{U}$, and $y\in\R^k$ is a regular value for $f$ in $U$. We then 
proceed to the definition of degree in the general case of \emph{admissible triples} 
when $f$ is assumed only continuous and proper on $\cl{U}$, and $y$ is any point 
in $\R^k\setminus f(\partial U)$. Lastly, we consider the so-called extended case of
the \emph{weakly admissible triples}, that is when $f$ is defined (and continuous) 
at least on $U$ and $y\in\R^k$ is such that $f^{-1}(y)\cap U$ is compact.

Our approach emphasizes the importance of three fundamental properties of topological
degree: Normalization, Additivity, and Homotopy Invariance (see below). Actually, 
these properties determine the notion of degree in a unique way yielding a computation
formula for the degree valid for admissible triples $(f,U,y)$ such that $f$ is 
Fr\'echet differentiable in any $x\in f^{-1}(y)$. This allows an alternative approach.

This paper is organized as follows: Section \ref{pre} gathers some results and notions 
needed for the following sections. In Section \ref{BdEs} the notion of degree both for 
regular triples and for admissible triples is defined, and the main consequences of 
the three above mentioned fundamental properties are explored. Section \ref{excase} is 
devoted to the notion of degree for weakly admissible triples, while Section 
\ref{appendix} contains the (lengthy) proof of the Homotopy Invariance Property for 
regular triples.
\smallskip

A word of caution:
unless differently stated, \emph{all the maps considered in this paper are continuous}.
We also recall that we say that a function defined on an arbitrary subset $X$ of $\R^k$
is $C^\infty$ (resp.\ $C^r$ with $1\leq r<\infty$), if it admits a $C^\infty$ (resp.\ 
$C^r$) extension to an open neighborhood of $X$.

\section{Preliminaries}\label{pre}

Let $f: U \to \R^s$ be a $C^1$ map defined on an open subset of $\R^k$. An
element $x \in U$ is called a \emph{critical point} (\emph{of} $f$) if the
Fr\'echet derivative $f'(x) \in \L(\R^k,\R^s)$ is not onto; otherwise $x$ is a
\emph{regular point}. An element $y \in \R^s$ is a \emph{critical value} if
$f^{-1}(y)$ contains critical points; otherwise $y$ is a \emph{regular value}. 

To avoid confusion, \emph{points} are in the source space and \emph{values} in
the target space.

Observe that if $k < s$, then any $x \in U$ is a critical point.
Consequently, $f(U)$ coincides with the set of critical values of $f$.

A very important special case is when $k = s$. In this context, $x \in
U$ is a regular point (of $f$) if and only if the Jacobian of $f$ at $x$,
$\det(f'(x))$, is nonzero. When this holds, the sign of $\det(f'(x))$ is
called the \emph{index of $f$ at $x$} and denoted $\i(f,x)$. 
Actually, the index $\i(f,x)$ is defined as $\sign(\det(f'(x)))$ even if $f$
is simply continuous, provided it is Fr\'echet differentiable at $x$ with
invertible derivative.

\begin{exercise}
\label{exercise_polynomial_map}
Let $p$ be a complex polynomial and regard $p$ as a map from $\R^2$ into
itself. Show that $z \in \C \cong \R^2$ is a critical point of $p$ if and
only if it is a root of the polynomial $p'$. Prove that $\i(p,z)=1$ 
for any regular point $z \in \C$.
\end{exercise}

The following result is of crucial importance in degree theory. (See e.g.,
\cite{Sar} or \cite{Mil}.) 

\begin{result}{Sard's Lemma}
Let $f: U \to \R^s$ be a $C^n$ map defined on an open subset of $\R^k$.
If $n > \max\{0, k-s\}$, then the set of critical values of $f$ has
($s$-dimensional) Lebesgue measure zero. In particular, the set of regular
values of $f$ is dense in $\R^s$.
\end{result}

Observe that, in view of Sard's Lemma, a $C^1$ curve $\alpha: [a,b] \to \R^s$,
$s > 1$, cannot be a Peano curve (i.e.\ a curve whose image contains interior
points).

\begin{definition}
\label{definition_proper}
A map $f: X \to Y$ between two metric spaces is \emph{proper} if
$f^{-1}(K)$ is compact whenever $K \subseteq Y$ is compact.
\end{definition}

Clearly, if $X$ is compact, then $f$ is proper (any map is assumed to be
continuous).

\begin{exercise}
\label{exercise_proper_is_closed}
Show that if $f: X \to Y$ is proper, then it is a closed map (that
is, $f(A)$ is closed whenever $A \subseteq X$ is closed).
\end{exercise}

\begin{exercise}
\label{exercise_proper_map}
Let $X \subseteq \R^k$ be closed and unbounded. Prove that a map $f: X \to
\R^s$ is proper if and only if
\[
\lim_{x \in X,\; |x| \to +\infty} |f(x)| = +\infty.
\]
\end{exercise}

\begin{example}
\label{example_polynomial_is_proper}
Let $p: \C \to \C$ be a non-constant complex polynomial. Then
$
\lim_{|z| \to +\infty} |p(z)| = +\infty.
$
Thus, $p$ is a proper map.
\end{example}

\medskip
In the following we will need to approximate continuous functions with
more regular ones. To do that, we shall make use of the following approximation 
theorem

\begin{result}{Smooth Approximation Theorem}
Let $U\subseteq\R^k$ be open, and let $f$ be an $\R^s$-valued (continuous) function 
defined on the closure $\cl{U}$ of $U$ in $\R^k$. Then, given  a continuous
function $\varepsilon:\cl{U}\to(0,\infty)$, there exists a $C^\infty$ function 
$g:\cl{U}\to\R^s$ such that $|f(x)-g(x)|<\varepsilon(x)$ for any $x\in\cl{U}$.
\end{result}

This fact could be proved directly. However, since any continuous function defined 
on a closed subset of $\R^k$ with values in $\R^s$ can be extended to a continuous 
function on $\R^k$ (this is a consequence of the well-known Tietze extension 
Theorem, see e.g., \cite{Dug}), the approximation result just stated can be deduced 
from more known theorems valid for maps defined on open sets, see e.g., 
\cite{Hir}.

\section{Brouwer degree in Euclidean spaces}\label{BdEs}
\subsection{The special case}

Let $U$ be an open subset of $\R^k$, $f$ an $\R^k$\mbox{-}valued map defined
(at least) on the closure $\cl U$ of $U$, and $y \in \R^k$.

\begin{definition}
\label{definition_admissible_triple}
The triple $(f,U,y)$ is said to be \emph{admissible} (for the Brouwer degree
in $\R^k$) provided that $f$ is proper on $\cl U$ and $f(x) \neq y$,
$\forall x \in \partial U$.
\end{definition}

Notice that, according to Exercise \ref{exercise_proper_is_closed},
$f(\partial U)$ is a closed subset of $\R^k$.

\begin{definition}
\label{definition_regular_triple}
An admissible triple $(f,U,y)$ is said to be \emph{regular} if
$f$ is $C^\infty$, and $y$ is a regular value for $f$ in $U$.
\end{definition}

We point out that if $(f,U,y)$ is a regular triple, then the set $f^{-1}(y)
\cap U$ is finite. In fact, $f^{-1}(y) \cap \cl U$ is compact ($f$ being
proper on $\cl U$), it is contained in $U$ (since $y \notin f(\partial
U)$) and it is discrete (because of the Inverse Function Theorem). 
This justifies the following definition of degree for the special case of a
regular triple.

\begin{definition}
\label{definition_degree_regular}
The \emph{Brouwer degree} of a regular triple $(f,U,y)$ is the integer
\begin{equation}
\label{formula_degree_regular}
\deg(f,U,y)\; := \sum_{x \in f^{-1}(y) \cap U} \i(f,x)
\end{equation}
\end{definition}

In some sense the Brouwer degree of a regular triple $(f,U,y)$ is an algebraic
count of the number of solutions in $U$ of the equation $f(x) = y$. This
integer, as we shall see, turns out to depend only on the connected
component of $\R^k \setminus f(\partial U)$ containing the regular value $y$.
This is not so for the absolute count of the solutions (i.e.\ the cardinality
$\#f^{-1}(y)$ of the set $f^{-1}(y)$), as it happens, for example, to the
proper map $f: \R \to \R$ given by $f(x) = x^2$. Incidentally, observe that
in this case we have $\deg(f,\R,y) = 0$ for any regular value $y \in \R$
(i.e.\ for any $y \neq 0$). 

Notice that the notation $\deg(f,U,y)$ is not redundant, since $\cl U$
can be strictly contained in the domain of $f$ (which is uniquely associated
with $f$). For example, if $\deg(f,U,y)$ is defined and $V$ is an
open subset of $U$ such that $f^{-1}(y) \cap \partial V = \emptyset$, then
also $\deg(f,V,y)$ is defined (and depends only on the restriction of $f$ to
$V$). 

Observe also that $(f,U,y)$ is a regular triple if and only if so is
$(f-y,U,0)$,  where $f-y$ stands for the map $x \mapsto f(x)-y$. Obviously,
when this holds, one has
\[
\deg(f,U,y) = \deg(f-y,U,0).
\]

\begin{example}
Given a positive integer $n$, let $p_n: \C \to \C$ be the map defined
by $p_n(z) = z^n$. Identifying $\C$ with $\R^2$, $0 \in \C$ is the only
critical point of $p_n$ (see Exercise \ref{exercise_polynomial_map}).
Therefore, $0 = p_n(0)$ is the unique critical value and, consequently,
recalling that $p_n$ is a proper map (see Exercise
\ref{exercise_proper_map}), $\deg(p_n,\C,w)$ is defined for any $w \in \C
\setminus \{0\}$. Since any $w \neq 0$ admits exactly $n$ different
$n$\mbox{-}roots, Exercise \ref{exercise_polynomial_map} shows that
$\deg(p_n,\C,w) = n$ for all $w \neq 0$. It is therefore natural to extend
the function $w \mapsto \deg(p_n,\C,w)$ by putting $\deg(p_n,\C,w) = n$ even
for $w = 0$ (this will be a consequence of the general definition of degree).
\end{example}

Theorem \ref{theorem_fundamental_properties_regular} below collects \emph{the
three fundamental properties of the degree for regular triples}. The first
two, the \emph{Normalization} and the \emph{Additivity}, are a straightforward
consequence of the definition; the third one, the \emph{Homotopy Invariance},
is crucial for the construction of the degree in the general case, is
nontrivial and (to please the impatient reader) will be proved in the appendix
to this chapter. As we shall see later, there exists at most one
integer-valued function (defined on the set of all regular triples)
satisfying the three fundamental properties.

In order to simplify the statement of Theorem
\ref{theorem_fundamental_properties_regular}, it is convenient to introduce
the following notion.

\begin{definition}
\label{definition_admissible_homotopy}
Let $U$ be an open subset of $\R^k$, $H$ an $\R^k$\mbox{-}valued map defined
(at least) on $\cl U \times [0,1]$, and $\alpha: [0,1] \to \R^k$ a path.
The triple $(H,U,\alpha)$ is said to be a \emph{homotopy of triples} (on
$U$, joining $(H(\cdot,0),U,\alpha(0))$ with $(H(\cdot,1),U,\alpha(1))$\,).
If, in addition, $H$ is proper on $\cl U \times [0,1]$ and $H(x,\lambda)
\neq \alpha(\lambda)$ for all $(x,\lambda) \in \partial U \times [0,1]$, then
$(H,U,\alpha)$ is called an \emph{admissible homotopy} (of triples).
If both $H$ and $\alpha$ are smooth maps, then $(H,U,\alpha)$ is said to be
\emph{smooth}.
\end{definition}

\begin{theorem}
\label{theorem_fundamental_properties_regular}
The degree for regular triples satisfies the following three fundamental
properties:
\begin{itemize}
\item[] \emph{(Normalization)} $\deg(I,\R^k,0) = 1$, where $I$
denotes the identity on $\R^k$;
\item[] \emph{(Additivity)} if $(f,U,y)$ is regular, and $U_1$ and $U_2$ are
two disjoint open subsets of $U$ such that $f^{-1}(y) \cap U \subseteq U_1
\cup U_2$, then
\[
\deg(f,U,y) = \deg(f,U_1,y) + \deg(f,U_2,y);
\]
\item[] \emph{(Homotopy Invariance)} if $(H,U,\alpha)$ is a smooth
admissible homotopy joining two regular triples, then
\[
\deg(H(\cdot,0),U,\alpha(0)) = \deg(H(\cdot,1),U,\alpha(1)).
\]
\end{itemize}
\end{theorem}

\subsection{The general case}
The Brouwer degree, preliminarily defined for regular triples, can be
extended to the larger class of admissible triples; where, we recall, a
triple $(f,U,y)$ is admissible (for the degree in Euclidean spaces) provided
that $U$ is an open subset of $\R^k$, $f$ is an $\R^k$\mbox{-}valued map which
is proper on $\cl U$, and $y \in \R^k$ does not belong to the (possibly
empty) set $f(\partial U)$.

The passage from the regular to the admissible case can be made in one big
step or, as usual, in two small steps (the intermediate stage regarding
admissible triples $(f,U,y)$ with $f$ smooth). We will reach the goal in just
one step, but in a way that the reader interested only in the smooth case can
easily imagine how to perform the first small step, which consists in
removing the assumption that the value $y$ in the triple $(f,U,y)$ is
regular.

\medskip
Before giving the definition of degree in the general case, we need some
preliminaries.

Let $f$ and $g$ be two $\R^s$\mbox{-}valued maps defined (at least) on a
subset $X$ of $\R^k$. Given $\epsilon > 0$, we say that \emph{$f$ is
$\epsilon$\mbox{-}close to $g$ in $X$} if  $|f(x) - g(x)| \le
\epsilon$, $\forall x \in X$. Moreover, given $y, z \in \R^s$, \emph{$y$ is 
$\epsilon$\mbox{-}close to $z$}, provided that $|y-z| \le \epsilon$.

Observe that if $(f,U,y)$ is an admissible triple and $g$ is
$\epsilon$\mbox{-}close to $f$ in $\cl U$ for some $\epsilon > 0$, then
$g$ is proper on $\cl U$ (see Exercise \ref{exercise_proper_map}). If,
in addition, $\epsilon < \dist(y,f(\partial U))$,\footnote{Recall the
convention $\inf \emptyset = +\infty$, which implies $\dist(y,\emptyset) =
+\infty$} then $y \notin g(\partial U)$, and in this case also the triple
$(g,U,y)$ is admissible. More generally, if $z \in \R^k$ is
$\sigma$\mbox{-}close to $y$ and $\epsilon + \sigma < \dist(y,f(\partial
U))$, then $(g,U,z)$ is admissible as well.

\begin{definition}
\label{definition_degree_admissible}
The degree of an admissible triple $(f,U,y)$, also called \emph{degree
of $f$ in $U$ at $y$}, is the integer
\[
\deg(f,U,y) := \deg(g,U,z),
\]
where $(g,U,z)$ is any regular triple with the following properties:
\begin{enumerate}
\item $g$ is $\epsilon$\mbox{-}close to $f$;
\item $z$ is $\sigma$\mbox{-}close to $y$;
\item $\epsilon + \sigma < \dist(y,f(\partial U))$.
\end{enumerate}
\end{definition}

Clearly, given $(f,U,y)$ admissible, the existence of a regular triple
$(g,U,z)$ as in Definition \ref{definition_degree_admissible} is ensured
by the Smooth Approximation Theorem (which shows the existence of $g$) and
Sard's Lemma (which shows the existence of $z$).

The following consequence of Theorem
\ref{theorem_fundamental_properties_regular} guarantees that this
Definition is actually well posed.

\begin{corollary}
\label{corollary_locally_constant}
Let $(f,U,y)$ be an admissible triple. Then,
\[
\deg(g_0,U,z_0) = \deg(g_1,U,z_1)
\]
for any pair of regular triples $(g_0,U,z_0)$ and $(g_1,U,z_1)$ satisfying
the following conditions:
\begin{enumerate}
\item $g_0$ and $g_1$ are $\epsilon$\mbox{-}close to $f$;
\item $z_0$ and $z_1$ are $\sigma$\mbox{-}close to $y$;
\item $\epsilon + \sigma < \dist(y,f(\partial U))$.
\end{enumerate}
\end{corollary}

\begin{proof}
Let $(g_0,U,z_0)$ and $(g_1,U,z_1)$ be as in the statement, and define the
smooth homotopy of triples $(H,U,\alpha)$ joining $(g_0,U,z_0)$ and
$(g_1,U,z_1)$ by
\[
H(x,\lambda) = (1-\lambda)g_0(x) +\lambda g_1(x), \quad
\alpha(\lambda) = (1-\lambda)z_0 +\lambda z_1.
\]
We have
\[
H(x,\lambda) - f(x) = (1-\lambda)(g_0(x) - f(x)) +\lambda(g_1(x) - f(x)).
\]
Thus,
\[
|H(x,\lambda) - f(x)| \le \epsilon,
\quad \forall (x,\lambda) \in \cl U \times [0,1],
\]
which implies that $H$ is proper on $\cl U$, on the basis of Exercise
\ref{exercise_proper_map}. Analogously,
\[
|\alpha(\lambda) - y| \le \sigma, \quad \forall \lambda \in [0,1].
\]
Let us show that
\[
H(x,\lambda) \neq \alpha(\lambda), \quad \forall (x,\lambda) \in \partial U
\times [0,1].
\]
In fact, given $(x,\lambda) \in \partial U \times [0,1]$, we have
\[
H(x,\lambda) - \alpha(\lambda) = H(x,\lambda) - f(x) + f(x) -
y + y - \alpha(\lambda)
\] 
and, consequently,
\[
|H(x,\lambda) - \alpha(\lambda)| \ge |f(x) - y| - \epsilon - \sigma > 0.
\]
The assertion now follows from the Homotopy Invariance Property for
regular triples (see Theorem \ref{theorem_fundamental_properties_regular}).
\end{proof}

The following important result is an extension, a consequence, and the
analogue of Theorem \ref{theorem_fundamental_properties_regular} for the
general case.

\begin{theorem}
\label{theorem_fundamental_properties_admissible}
The Brouwer degree in $\R^k$ satisfies the following three Fundamental
Properties:
\begin{itemize}
\item[] \emph{(Normalization)} $\deg(I,\R^k,0) = 1$, where $I$ denotes the
identity on $\R^k$;
\item[] \emph{(Additivity)} if $(f,U,y)$ is admissible, and $U_1$ and $U_2$
are two disjoint open subsets of $U$ such that $f^{-1}(y) \cap U \subseteq U_1
\cup U_2$, then
\[
\deg(f,U,y) = \deg(f,U_1,y) + \deg(f,U_2,y);
\]
\item[] \emph{(Homotopy Invariance)} if $(H,U,\alpha)$ is an admissible
homotopy, then
\[
\deg(H(\cdot,0),U,\alpha(0)) = \deg(H(\cdot,1),U,\alpha(1)).
\]
\end{itemize}
\end{theorem}

\begin{proof} Only the last two properties need to be proved.

\smallskip
(Additivity) Since $f$ if proper on $\cl U$, the subset
$C = f(\cl U \setminus (U_1 \cup U_2))$ of $\R^k$ is closed. Moreover,
the assumption $f^{-1}(y) \cap U \subseteq U_1 \cup U_2$ implies $\dist(y,C) >
0$. Let $g$ be any smooth map which is $\epsilon$\mbox{-}close to $f$, with
$\epsilon < \dist(y,C)$. It is easy to check that $g^{-1}(y) \cap U \subseteq
U_1 \cup U_2$. The assertion now follows from Definition
\ref{definition_degree_admissible} and the Additivity Property of the degree
for regular triples (stated in Theorem
\ref{theorem_fundamental_properties_regular}).

\smallskip
(Homotopy Invariance) Observe that, on the basis of Exercise
\ref{exercise_proper_map}, the map
\[
(x,\lambda) \mapsto H(x,\lambda) - \alpha(\lambda)
\]
is proper on $\cl U \times [0,1]$. Thus the image, under this map, of
the set $\partial U \times [0,1]$ is closed in $\R^k$. Consequently, since
this set does not contain the origin of $\R^k$, the extended real number
\[
\delta := \inf\Big\{ |H(x,\lambda) - \alpha(\lambda)| : (x,\lambda) \in
\partial U \times (0,1) \Big\}
\]
is nonzero.
Let $(G,U,\beta)$ be any smooth homotopy of triples, joining two regular
triples, and satisfying the following properties:
\begin{enumerate}
\item $G$ is $\epsilon$\mbox{-}close to $H$ (on $\cl U \times [0,1]$);
\item $\beta$ is $\sigma$\mbox{-}close to $\alpha$ (on $[0,1]$);
\item $\epsilon + \sigma < \delta$.
\end{enumerate}
The existence of such a triple is ensured by the Smooth Approximation Theorem
and Sard's Lemma. As in the proof of Corollary
\ref{corollary_locally_constant} one can show that $(G,U,\beta)$ is an
admissible homotopy. Therefore, because of the Homotopy Invariance Property
for regular triples, we get
\[
\deg(G(\cdot,0),U,\beta(0)) = \deg(G(\cdot,1),U,\beta(1)).
\]  
The assertion now follows from Definition
\ref{definition_degree_admissible}.
\end{proof}

\subsection{Direct consequences of the Fundamental Properties}
We will prove now some important additional properties of the Brouwer degree.
Even if they could be easily deduced from the definition of
degree, we prefer to prove them starting from the three \emph{Fundamental
Properties} stated in Theorem \ref{theorem_fundamental_properties_admissible}:
\emph{Normalization}, \emph{Additivity} and \emph{Homotopy Invariance}. The
advantage of this method will be evident in the next  subsection, which is
devoted to the axiomatic approach.

\smallskip
First of all we observe that, given a map $f: X \to \R^k$ defined on a subset
$X$ of $\R^k$ and given $y \in \R^k$, the triple $(f,\emptyset,y)$ is
admissible. Therefore, $\deg(f,\emptyset,y)$ is defined. We claim that 
this degree is zero.

Indeed, from the Additivity Property, putting $U = \emptyset$, $U_1 =
\emptyset$ and $U_2 = \emptyset$, we get
\[
\deg(f,\emptyset,y) = \deg(f,\emptyset,y) + \deg(f,\emptyset,y),
\]
which implies our assertion.

\medskip
The following property, which is evident in the regular case, shows that the
degree of an admissible triple $(f,U,y)$ depends only on the behavior of $f$
in any neighborhood of the set of solutions of the equation $f(x) = y$, $x
\in U$.

\begin{theorem}[Excision Property]
\label{theorem_excision}
If $(f,U,y)$ is admissible and $V$ is an open subset of $U$ such that
$f^{-1}(y) \cap U \subseteq V$, then $(f,V,y)$ is admissible and
\[
\deg(f,U,y) = \deg(f,V,y).
\]
\end{theorem}
\begin{proof}
The admissibility of $(f,V,y)$ is clear. To show the equality apply the
Additivity Property with $U_1 = V$ and $U_2 = \emptyset$.
\end{proof}

As the above property, also the following one is evident in the regular case. 

\begin{theorem}[Existence Property]
\label{theorem_existence}
If $\deg(f,U,y) \neq 0$, then the equation $f(x)
= y$ admits at least one solution in $U$.
\end{theorem}
\begin{proof}
Assume that $f^{-1}(y) \cap U$ is empty. By the Excision
Property, taking $V = \emptyset$, we get
\[
\deg(f,U,y) = \deg(f,\emptyset,y) = 0,
\]
which is a contradiction.
\end{proof}

Given an admissible triple $(f,U,y)$, since the target space of $f$ is $\R^k$,
the equation $f(x) = y$ is equivalent to $f(x) - y = 0$. In terms of degree
this fact is expressed by the following property, which is evident in the
regular case.

\begin{theorem}[Translation Invariance Property]
\label{theorem_translation_invariance}
If $(f,U,y)$ is admissible, then so is $(f-y,U,0)$, and
\[
\deg(f,U,y) = \deg(f-y,U,0).
\]
\end{theorem}
\begin{proof}
Consider the family of equations
\[
f(x) - \lambda y = (1 - \lambda) y, \quad \lambda \in [0,1],
\]
and apply the Homotopy Invariance Property with $H(x,\lambda) = f(x) -
\lambda y$ and $\alpha(\lambda) = (1 - \lambda) y$.
\end{proof}

The next result is a straightforward consequence of the Homotopy
Invariance Property, and the proof is left to the reader.

\begin{theorem}[Continuous Dependence Property]
\label{theorem_continuous_dependence}
Let $f$ be proper on the closure of an open subset $U$ of $\R^k$. Then the
map $y \mapsto \deg(f,U,y)$, which is defined on the open set $\R^k \setminus
f(\partial U)$, is locally constant. Thus, $\deg(f,U,y)$ depends only on the
connected component of  $\R^k \setminus f(\partial U)$ containing $y$.
\end{theorem}

Because of the above property, given an open $U \subseteq \R^k$, $f$
proper on $\cl U$ and a connected subset $V$ of $\R^k \setminus
f(\partial U)$, we will use the notation $\deg(f,U,V)$ to indicate the degree
of $f$ in $U$ at any $y \in V$. In particular, if $U = \R^k$, the integer
$\deg(f,\R^k,\R^k)$ will be denoted by $\deg(f)$.

\medskip
The following property means that, given $y \in \R^k$ and $U \subseteq \R^k$
open and \emph{bounded}, the degree of a map $f: \cl{U} \to \R^k$
(in $U$ at $y$) depends only on the restriction of $f$ to the boundary of $U$
(assuming the condition $y \notin f(\partial U)$, which, $U$ being bounded,
is sufficient for the degree to be defined). This is important since, in many
cases, it allows us to deduce the existence of solutions in $U$ of the
equation $f(x) = y$ only from the inspection of the behavior of $f$ along the
boundary of $U$; as in the case of $U = (a,b) \subseteq \R$, where the
condition $f(a)f(b) < 0$ implies $f(x) = 0$ has a solution in $(a,b)$.

\begin{theorem}[Boundary Dependence Property]
\label{theorem_boundary_dependence}
Let $U \subseteq \R^k$ be open and bounded, and let $f,g: \cl U \to
\R^k$ be such that $f(x) = g(x)$ for all $x \in \partial U$. Then, given $y
\in \R^k \setminus f(\partial U)$, one has
\[
\deg(f,U,y) = \deg(g,U,y).
\]
\end{theorem}
\begin{proof}
The assertion follows from the fact that the homotopy of triples
$(H,U,\alpha)$ defined by 
\[
H(x,\lambda) = \lambda f(x) + (1 - \lambda) g(x), \quad \alpha(\lambda) = y
\]
is admissible.
\end{proof}

We point out that, in the above result, the assumption that $U$ is bounded
cannot be dropped. To see this, take $U = (0,+\infty)$, $f(x) = x$, $g(x) =
-x$, $y=1$.

\subsection{The axiomatic approach}
{}From an axiomatic point of view, the topological degree (in Euclidean
spaces) is a map which to any admissible triple $(f,U,y)$ assigns an integer,
$\deg(f,U,y)$, satisfying the three \emph{Fundamental Properties} (stated in
Theorem \ref{theorem_fundamental_properties_admissible}):
\emph{Normalization}, \emph{Additivity} and \emph{Homotopy Invariance}.

A famous result by Amann-Weiss \cite{AmWe} (1973) asserts \emph{the uniqueness 
of the topological degree}. That is, there exists at most one integer-valued 
map (defined on the class of the admissible triples) which verifies the three
Fundamental Properties. 

There are several methods for the construction of degree (see, for
example, \cite{Dei,Dol,GuPo,Llo,Mil,Nag,Nir,Sch,Zei}), however, because of 
the Amann-Weiss result, with any of such
methods, what is important is to prove the three Fundamental Properties
(called, in this subsection, \emph{Amann-Weiss axioms}): all the other
classical properties will follow, as we have already shown in the previous
subsection.

\medskip
Let us show that from the three Amann-Weiss axioms one obtains an explicit
formula for computing the degree of triples which are, in a sense to be made
precise, dense in the family of the admissible triples. 
In particular, we will show that when an admissible triple $(f,U,y)$ is
actually regular, then
\[
\deg(f,U,y)\; = \sum_{x \in f^{-1}(y) \cap U} \i(f,x).
\]
The uniqueness of the degree will follow easily from the above formula and
the Homotopy Invariance Property (see Theorem \ref{theorem_degree_uniqueness}
below).

Recall first that, given a proper map $f: \R^k \to \R^k$, $\deg(f)$ stands for
$\deg(f,\R^k,y)$, where $y$ is any value in $\R^k$. This notation is
justified by the Continuous Dependence Property, which, as all the other
properties in the previous subsection, is a consequence of the axioms.
Observe that, because of the Existence Property, if $\deg(f) \neq 0$,
then $f$ is surjective. We point out also that, as a consequence of the
Homotopy Invariance axiom, if a homotopy $H:
\R^k
\times [0,1] \to
\R^k$ is proper, then
$\deg(H(\cdot,
\lambda))$ is well defined and independent of $\lambda$.

Let, as usual, $\L(\R^k)$ denote the normed space of linear endomorphisms of
$\R^k$ and let $\GL(\R^k)$ stand for the open subset of $\L(\R^k)$ of the
automorphisms; that is,
\[
\GL(\R^k) = \Big\{ L \in \L(\R^k): \det(L) \neq 0 \Big\}.
\]
Now, let $L \in \GL(\R^k)$ be given. Since $L$ is invertible, it is a proper
map of $\R^k$ onto itself (notice that any homeomorphism is a proper map).
Thus, $\deg(L)$ is well defined.

Let us show that the Amann-Weiss axioms imply
\begin{equation}
\label{formula_degree_endomorphism}
\deg(L) = \sign(\det(L)), \quad \forall L \in \GL(\R^k).
\end{equation}

To this end, we recall that the open subset $\GL(\R^k)$ of $\L(\R^k)$ has
exactly two connected components. Namely,
\[
\GL_+(\R^k) = \Big\{ L \in \L(\R^k): \det(L) > 0 \Big\}.
\]
and
\[
\GL_-(\R^k) = \Big\{ L \in \L(\R^k): \det(L) < 0 \Big\}.
\]
As a consequence of the Homotopy Invariance axiom it is easy to check that
the map which assigns $\deg(L)$ to any $L \in \GL(\R^k)$ is locally constant.
Indeed, if $L_0$ and $L_1$ are close one to the other, the homotopy
$H(x,\lambda) = L_0 x + \lambda(L_1 - L_0)$ is proper. Consequently,
$\deg(L)$ depends only on the component of $\GL(\R^k)$ containing $L$.

Since the identity $I$ of $\R^k$ belongs to $\GL_+(\R^k)$, the Normalization
axiom implies $\deg(L) = 1$, $\forall L \in \GL_+(\R^k)$. 

Let us show that $\deg(L) = -1$, $\forall L \in \GL_-(\R^k)$. For this purpose
consider the map $f: \R^k \to \R^k$ given by
\[
f(\xi_1,\ldots,\xi_{k-1},\xi_k) = (\xi_1,\ldots,\xi_{k-1},|\xi_k|).
\]
This map is proper, since $\|f(x)\| = \|x\|$, $\forall x \in \R^k$. Thus
$\deg(f)$ makes sense and is zero, because $f$ is not surjective.

Let $V_-$ and $V_+$ denote, respectively, the open half-spaces of the points
in $\R^k$ with negative and positive last coordinate. Consider the two
solutions
\[
x_- = (0,\ldots,0,-1) \quad \mbox{and} \quad x_+ = (0,\ldots,0,1)
\]
of the equation $f(x) = y$, with $y = (0,\ldots,0,1)$, and observe that $x_-
\in V_-$, $x_+ \in V_+$.

By the Additivity axiom we get
\[
0 = \deg(f) = \deg(f,V_-,y) + \deg(f,V_+,y).
\]
Now, observe that in $V_+$ the map $f$ coincides with the identity $I$ of
$\R^k$. Therefore, because of the Excision Property, one has
\[
\deg(f,V_+,y) = \deg(I) = 1,
\]
which implies $\deg(f,V_-,y) = -1$.

Since $f$ in $V_-$ coincides with the linear map $L_- \in
\GL_-(\R^k)$ given by
\[
(\xi_1,\ldots,\xi_{k-1},\xi_k) \mapsto (\xi_1,\ldots,\xi_{k-1},-\xi_k),
\]
we obtain $\deg(L_-) = -1$. Thus, $\GL_-(\R^k)$ being connected, we finally
get $\deg(L) = -1$ for all $L \in \GL_-(\R^k)$, as claimed.

\medskip
Let us show how from the Amann-Weiss axioms one can deduce the formula
\eqref{formula_degree_regular} for computing the degree of a regular triple.
More generally, we prove the following result.

\begin{theorem}[Computation Formula]
\label{theorem_computation_formula}
Let $(f,U,y)$ be an admissible triple. Assume that, at any $x \in
f^{-1}(y) \cap U$, $f$ is Fr\'echet differentiable with nonsingular
derivative. Then $f^{-1}(y) \cap U$ is finite and 
\begin{equation*}
\label{formula_degree_general}
\deg(f,U,y)\; = \sum_{x \in f^{-1}(y) \cap U} \i(f,x).
\end{equation*}
\end{theorem}

In order to prove Theorem \ref{theorem_computation_formula} we need the following
result.

\begin{lemma}
\label{lemma_degree_linearization}
Let $(f,V,y_0)$ be an admissible triple. Assume that the equation $f(x) = y_0$
has a unique solution $x_0 \in V$. If $f$ is Fr\'echet differentiable at $x_0$
and $f'(x_0)$ is invertible, then $\deg(f,V,y) = \deg(f'(x_0))$.
\end{lemma}
 
\begin{proof}
Since $f$ is differentiable at $x_0$, we have
\[
f(x) = y_0 + f'(x_0)(x-x_0) + \|x-x_0\|\epsilon(x-x_0),\quad \forall x \in
\cl V,
\]
where $\epsilon(h)$ is defined for $h \in - x_0 + \cl V$, is continuous,
and such that $\epsilon(0) = 0$.

Observe that the linearized map of $f$ at $x_0$,
$g(x) := y_0 + f'(x_0)(x-x_0)$, is an affine map with linear part
$f'(x_0) \in \GL(\R^k)$. Thus, $g$ is proper and, because of the
Translation Invariance Property, one has $\deg(g) = \deg(f'(x_0))$.
Therefore, by the Excision Property, it is enough to show that
\begin{equation}
\label{equation_degree_linearization}
\deg(f,W,y_0) = \deg(g,W,y_0),
\end{equation}
where $W$ is a sufficiently small open neighborhood of
$x_0$ contained in $V$.

For this purpose, define the homotopy $H: \cl V \times [0,1] \to \R^k$
joining
$g$ with
$f$ by 
\[
H(x,\lambda) = y_0 + f'(x_0)(x-x_0) + \lambda\|x-x_0\|\epsilon(x-x_0).
\]
We have
\[
\|H(x,\lambda) - y_0\| \ge \left(m - \|\epsilon(x - x_0)\|\right)\|x - x_0\|,
\]
where $m = \inf\{ \|f'(x_0)v\|: \|v\| = 1\}$ is positive, $f'(x_0)$ being
invertible. This shows that, in a convenient neighborhood $W$ of $x_0$, the
homotopy of triples $(H,W,y_0)$ is admissible, and the equality
\ref{equation_degree_linearization} is established.
\end{proof}

\begin{proof}[Proof of Theorem \ref{theorem_computation_formula}]
Since $f$ is proper on $\cl U$, the set $f^{-1}(y)
\cap \cl U$ is compact, and the condition $y \notin f(\partial
U)$ ensures that it is contained in $U$. On the other hand, as in the proof
of Lemma \ref{lemma_degree_linearization}, the assumption that at any $x \in
f^{-1}(y) \cap \cl U$ the derivative $f'(x)$ is injective ensures that
this set is made up of isolated points. Therefore, it is actually a finite
set. Let $V_1, V_2, \ldots, V_n$ be pairwise disjoint open subsets of $U$,
each of them containing exactly one point of $f^{-1}(y)$. The Additivity
axiom implies
\[
\deg(f,U,y) = \sum_{i=1}^n \deg(f,V_i,y),
\]
and the assertion follows from Lemma \ref{lemma_degree_linearization} and
formula \eqref{formula_degree_endomorphism} for computing the degree of a
linear automorphism.
\end{proof}

The following result shows, in particular, that the degree of an admissible
triple coincides with the degree of any sufficiently close regular
triple, and this implies the uniqueness of the degree.

\begin{theorem}
\label{theorem_degree_uniqueness}
Let $(f,U,y)$ be an admissible triple. If $(g,U,z)$ is a regular
triple which can be joined to $(f,U,y)$ via an admissible homotopy, then
\begin{equation}
\label{equation_degree_uniqueness}
\deg(f,U,y)\; =  \sum_{x \in g^{-1}(z) \cap U} \i(g,x).
\end{equation}
In particular, this is true for any regular triple $(g,U,z)$ such that:
\begin{enumerate}
\item $g$ is $\epsilon$\mbox{-}close to $f$;
\item $z$ is $\sigma$\mbox{-}close to $y$;
\item $\epsilon + \sigma < \dist(y,f(\partial U))$.
\end{enumerate}
\end{theorem}

\begin{proof}
Let $(g,U,z)$ be a regular triple which can be joined to $(f,U,y)$ via an
admissible homotopy. {}From the Homotopy Invariance axiom one gets
\[
\deg(f,U,y) = \deg(g,U,z),
\]
and the equality \eqref{equation_degree_uniqueness} follows from Theorem
\ref{theorem_computation_formula}.

Assume now that $(g,U,z)$ is a regular triple which satisfies properties $1$,
$2$ and $3$. It is enough to show that the homotopy of triples
$(H,U,\alpha)$, defined by
\[
H(x,\lambda) = (1-\lambda) f(x) + \lambda g(x), \quad \alpha(\lambda) =
(1-\lambda) y + \lambda z,
\]
is admissible on $U$. This can be done as in the proof of Corollary
\ref{corollary_locally_constant}.
\end{proof}

\subsection{First topological applications}
We give now some direct topological applications of the Brouwer degree in
Euclidean spaces.

\smallskip
Let us show, first of all, that the topological degree of a (non-constant)
polynomial is the same as its algebraic degree. This provides one of the many
proofs of the Fundamental Theorem of Algebra (which is actually a result of
topological nature) and justifies the expression ``degree'' used by Brouwer.
In some sense, the Brouwer degree is an extension of the algebraic notion of
degree to more general situations.

\smallskip
As before, if $f: \R^k \to \R^k$ is a proper map, by $\deg(f)$ we shall mean
the integer $\deg(f,\R^k,y)$, where $y$ is any point of $\R^k$. Because of the
Continuous Dependence Property (Theorem \ref{theorem_continuous_dependence}),
$\deg(f)$ is well defined.

Observe that if $H: \R^k \times [0,1] \to \R^k$ is a proper map, then, because
of the Homotopy Invariance Property, $\deg(H(\cdot,\lambda))$ is independent
of $\lambda \in [0,1]$.

\begin{theorem}
\label{theorem_degree_polynomial}
Let $p_n: \C \to \C$ be a polynomial of algebraic degree $n > 0$. Then
$p_n$, regarded as a map from $\R^2$ into itself, has topological degree $n$.
\end{theorem}

\begin{proof}
Write $p_n(z) = az^n + q(z)$, with $a \neq 0$ and $q(z)$ a polynomial
of degree less then $n$. Consider the homotopy
\[
H(z,\lambda) = az^n + \lambda q(z)
\]
and observe that
\[
\lim_{|z| \to +\infty} |H(z,\lambda)| = +\infty,
\]
uniformly with respect to $\lambda \in [0,1]$. Thus $H$ is a proper map and,
consequently, the topological degrees of the two maps $p_n$ and
$f_n: z \mapsto az^n$ are equal. To conclude that the Brouwer degree of $f_n$
is $n$, observe that the equation $az^n = a$ has exactly $n$ solutions each
of them with index one (see Exercise \ref{exercise_polynomial_map}).
\end{proof}

With the same method as in the proof of Theorem
\ref{theorem_degree_polynomial} one can show that if $p_n: \C \to \C$ is a
polynomial of algebraic degree $n > 0$, then the map $f_n: \C \to \C$ given
by $f_n(z) = p_n(\bar z)$, where $\bar z$ is the conjugate of $z$, has
degree $-n$. Thus, we have examples of proper maps from $\R^2$ into $\R^2$ of
arbitrary nonzero degree. A simple (proper) map of degree zero is given by
$f_0(x,y) = (x,y^2)$.

It is easy to check that if $f: \R^2 \to \R^2$ is proper and $I$ denotes the
identity on $\R^s$, then the map
\[
I \times f: \R^s \times \R^2 \to \R^s \times \R^2
\]
is proper and $\deg(I \times f) = \deg(f)$. Thus, if $k > 1$, one can find
maps from $\R^k$ into itself of arbitrary degree.

\begin{exercise}
Show that if $f: \R \to \R$ is proper, then $\deg(f)$ may assume
 only three values: $-1, 0, 1$.
\end{exercise}

{}From Theorem \ref{theorem_degree_polynomial} and the Existence Property of
the degree follows immediately the

\medskip\noindent
\textbf{Fundamental Theorem of Algebra.} {\em Any non-constant polynomial
with complex coefficients admits at least one root.}

\medskip
The following is another famous topological result that can be easily deduced
from degree theory.

\medskip\noindent
\textbf{Brouwer Fixed Point Theorem.} {\em Let $U$ be the open unit ball in
$\R^k$ and let $f: \cl U \to \R^k$ be continuous and such that
$f(\cl U) \subseteq \cl U$ (or, more generally, $f(\partial
U) \subseteq \cl U$). Then $f$ has a fixed point in  $\cl U$.}

\begin{proof}
If the triple $(I-f,U,0)$ is not admissible, then $f$ has a fixed point on
$\partial U$, and we are done. Assume, therefore, this is not the case.
Consider the homotopy $H: \cl U \times [0,1] \to \R^k$ given by
$H(x,\lambda) = x - \lambda f(x)$ and observe that $x \neq \lambda f(x)$ for
all $x \in \partial U$ and $\lambda \in [0,1]$. Thus, by the Homotopy
Invariance Property, we have
\[
\deg(I-f,U,0) = \deg(I,U,0).
\]
On the other hand, the Excision Property implies
\[
\deg(I,U,0) = \deg(I,\R^k,0) = 1.
\]
The result now follows from the Existence Property applied to the equation $x
- f(x) = 0$, $x \in U$.
\end{proof}

We recall that a subset $A$ of a topological space $X$ is a \emph{retract} of
$X$ if there exists a continuous map $r: X \to A$, called \emph{retraction},
whose restriction to $A$ is the identity map. Clearly the boundary of an
interval $[a,b] \subseteq \R$, being disconnected, is not a retract of
$[a,b]$. The following easy consequence of the Boundary Dependence Property
extends this elementary fact.

\begin{theorem}
\label{theorem_no_retraction}
Let $U$ be a bounded open subset of $\R^k$. Then $\partial U$ is not a
retract of $\cl U$.
\end{theorem}

\begin{proof}
Assume there exists a map $r: \cl U \to \partial U$ such that $r(x) =
x$, $\forall x \in \partial U$. Let $y \in U$. Since $U$ is bounded, the
Boundary Dependence Property (Theorem
\ref{theorem_boundary_dependence}) implies
\[
\deg(r,U,y) = \deg(I,U,y) = 1.
\]
Hence the equation $r(x) = y$ has a solution in $U$, and this is a
contradiction since $r$ maps $\cl U$ onto $\partial U$.
\end{proof}

The fact that the boundary of the subset $(0,+\infty)$ of $\R$ is a
retract of $[0,+\infty)$ shows that in Theorem \ref{theorem_no_retraction} the
assumption that $U$ is bounded cannot be removed. 

\medskip
We give now some applications of degree theory to problems of existence and
multiplicity of solutions for nonlinear equations in $\R^k$.

We start with the following result.

\begin{proposition}
\label{proposition_proper_plus_bounded}
Let $f: \R^k \to \R^k$ be proper and let $g: \R^k \to \R^k$ be a bounded map.
Then $f + g$ is proper and $\deg(f+g) = \deg(f)$. Consequently, if $\deg(f)
\neq 0$, the equation $f(x) + g(x) = 0$ has at least one solution.
\end{proposition}

\begin{proof}
The map $f + g$ is proper (on the basis of Exercise
\ref{exercise_proper_map}), since the assumption
\[
\lim_{\|x\| \to +\infty}\|f(x)\| = +\infty
\]
implies
\[
\lim_{\|x\| \to +\infty} \|f(x) + g(x)\| = +\infty.
\]
For the same reason, also the map $(x,\lambda) \mapsto f(x) + \lambda
g(x)$ is proper, and the equality $\deg(f+g) = \deg(f)$ follows
immediately from the Homotopy Invariance Property.
\end{proof}

Let us show, with a simple example, how the above result can be applied to
prove the existence of solutions of a nonlinear equation in $\R^k$.
\begin{example}
Consider the following nonlinear system of two equations in two unknowns:
\begin{equation}
\label{equation_proper_plus_bounded}
\left\{
\begin{array}{lcc}
x^2 - 2y^2 - \sin(xy) & = & 0\\
xy + 2\cos x + \frac{1}{1+y^2} & = & 0
\end{array}
\right.
\end{equation}
We claim that this system has at least one solution. It is not difficult to
check that the map $f: \R^2 \to \R^2$ given by $f(x,y) = (x^2 - 2y^2,xy)$ is
proper. Indeed, let $c$ be a positive constant and consider the inequalities
\[
|x^2 - 2y^2| \le c \quad \mbox{and} \quad |xy| \le c.
\]
The first one implies that when $|x|$ is large, so is $|y|$; but this is in
contrast with the second one. Thus, if $f(x,y)$ belongs to a compact set,
also $(x,y)$ must stay in a compact set.

To compute the degree of $f$, which is well defined, observe that the system
\[
\left\{
\begin{array}{lcc}
x^2 - 2y^2 & = & 1\\
xy  & = & 0
\end{array}
\right.
\]
has the following two solutions: $(1,0)$ and $(-1,0)$. One can check that
these solutions are both regular with index $1$. Thus $\deg(f) = 2$, and
Proposition \ref{proposition_proper_plus_bounded} implies the existence of at
least one solution of system \eqref{equation_proper_plus_bounded}, as claimed.

Actually, since $\deg(f) = 2$, applying Sard's Lemma (and the definition of
degree for a regular triple) one can say more: for almost all $(a,b) \in \R^2$
the system
\[
\left\{
\begin{array}{lcc}
x^2 - 2y^2 - \sin(xy) & = & a\\
xy + 2\cos x + \frac{1}{1+y^2} & = & b
\end{array}
\right.
\]
has at least two solutions (and, of course, at least one for all $(a,b) \in
\R^2$).
\end{example}

\medskip
{}From Proposition \ref{proposition_proper_plus_bounded} it follows
immediately that a nonlinear system of the type
\begin{equation}
\label{equation_linear_plus_bounded}
Lx + g(x) = 0
\end{equation}
has at least one solution, provided that $L$ is an invertible linear operator
in $\R^k$ and $g: \R^k \to \R^k$ is a bounded map. This, on the other hand,
can be shown also by means of the Brouwer Fixed Point Theorem, since
\eqref{equation_linear_plus_bounded}
can be equivalently written as a fixed point equation in the form $x =
-L^{-1}g(x)$, where the map $x \mapsto -L^{-1}g(x)$ sends the whole space
$\R^k$ into a bounded set (and, in particular, some closed ball into itself).

An extension of this existence result is Corollary
\ref{corollary_linear_plus_nonlinear} below, which is a direct consequence of
the following continuation principle (stated without any notion of degree
theory).

\begin{theorem}[Continuation Principle in Euclidean Spaces]
\label{theorem_continuation_Euclidean}
Let $U$ be a bounded open subset of $\R^k$, $f: \cl U \to \R^k$ a
continuous map of class $C^1$ in a neighborhood of $f^{-1}(0)$, and $h:
\cl U
\times [0,1]
\to
\R^k$ a continuous map. Assume that:
\begin{enumerate}
\item $h(x,0) = 0$ for all $x \in U$;
\item $f(x) + h(x,\lambda) \neq 0$ for all $(x,\lambda) \in \partial U
\times [0,1]$;
\item $\det(f'(x)) \neq 0$ for any $x \in f^{-1}(0)$;
\item the integer
\[
\sum_{x \in f^{-1}(0)} \sign(\det(f'(x)))
\]
is nonzero.
\end{enumerate}
Then, the equation $f(x) + h(x,1) = 0$ has at least one solution in $U$.
\end{theorem}

\begin{proof}
Define $H: \cl U \times [0,1] \to \R^k$ by $H(x,\lambda) =
f(x) + h(x,\lambda)$ and observe that, because of assumption $2$, the triple
$(H,U,0)$ is an admissible homotopy. Thus, from the Homotopy Invariance
Property it follows that
\[
\deg(H(\cdot,0),U,0) = \deg(H(\cdot,1),U,0).
\]
On the other hand, because of condition $1$, we have $H(\cdot,0) = f$, and
the assertion now follows from assumptions $3$ and $4$, the Computation
Formula (Theorem \ref{theorem_computation_formula}), and the Existence Property
(Theorem \ref{theorem_existence}).
\end{proof}

The following easy consequence of Theorem
\ref{theorem_continuation_Euclidean} extends the existence result related to
equation \eqref{equation_linear_plus_bounded}, removing the assumption that
the map $g$ is bounded.

\begin{corollary}
\label{corollary_linear_plus_nonlinear}
Let $L$ be a linear operator in $\R^k$ and $g: \R^k \to \R^k$ a continuous
map. If the set
\[
S = \Big\{ x \in \R^k: Lx + \lambda g(x) = 0 \mbox{ for some } \lambda \in
[0,1] \Big\}
\]
is bounded, then the equation $Lx + g(x) = 0$ has at least one solution.
\end{corollary}

\begin{proof}
Observe that the boundedness of $S$ implies that $L$ is injective and,
consequently, $\det(L) \neq 0$. The assertion now follows from Theorem
\ref{theorem_continuation_Euclidean} with $U$ any open ball containing $S$,
$f = L$, and $h(x,\lambda) = \lambda g(x)$.
\end{proof}

Corollary \ref{corollary_linear_plus_nonlinear} can be proved in a more
elementary way: it can be deduced directly from the Brouwer Fixed
Point Theorem (this is not so for the above continuation principle). Let us
show, briefly, how this can be done.

Define $\sigma: \R^k \to [0,1]$ by $\sigma(x) = \max\{1-\dist(x,S),0\}$ and
observe that any solution of the equation $Lx + \sigma(x)g(x) = 0$ lies in
$S$. Therefore it is also a solution of $Lx + g(x) = 0$, since $\sigma(x) =
1$ for $x \in S$. To show that $Lx + \sigma(x)g(x) = 0$ has a solution, notice
that $L$ is invertible (since $S$ is bounded) and apply the Brouwer Fixed
Point Theorem to the equation $x = -\sigma(x)L^{-1}g(x)$.

\begin{example}
\label{example_linear_plus_nonlinear}
To illustrate how Corollary \ref{corollary_linear_plus_nonlinear} applies,
we prove that the system
\begin{equation*}
\left\{
\begin{array}{lll}
x + y + x^3  + \sin(xy) & = & 0\\
y + 2\cos(xy) + y^5 & = & 0
\end{array}
\right.
\end{equation*}
has at least one solution. For this purpose we need to show that all the
possible solutions $(x,y)$ of the system
\begin{equation*}
\left\{
\begin{array}{lll}
x + y + \lambda x^3  + \lambda\sin(xy) & = & 0\\
y + 2\lambda\cos(xy) + \lambda y^5 & = & 0
\end{array}
\right.
\end{equation*}
are \emph{a priori} bounded when the parameter $\lambda$ varies in $[0,1]$.
In fact, the second equation implies that, if $(x,y)$ is such a solution,
then $y$ must lie in the interval $[-2,2]$ and, as a consequence, from the
first equation one gets $|x| \le 3$.
\end{example}

Degree can be useful to prove the existence of nontrivial solutions of an
equation of the type $f(x) = 0$, where $f: \R^k \to \R^k$ satisfies the
condition $f(0) = 0$. The following result is in this direction.

\begin{theorem}
\label{theorem_nontrivial_solution}
Let $f: \R^k \to \R^k$ be a proper map such that $f(0) = 0$. Assume that $f$
is Fr\'echet differentiable at the origin. If $f'(0)$ is invertible and
$\deg(f) \neq \i(f,0)$, then the equation $f(x) = 0$ has a nontrivial solution 
(i.e.\ a solution $x \neq 0$).
\end{theorem}

\begin{proof}
If the equation $f(x) = 0$ had the unique solution $x = 0$, the Computation
Formula would contradict the assumption $\deg(f) \neq \i(f,0)$.
\end{proof}

\begin{example}
Consider the system
\[
\left\{
\begin{array}{lll}
x - 2\sin(x + x^2 - y^2) & = & 0\\
2x + y + 1 - \cos(xy)    & = & 0
\end{array}
\right.
\]
and observe that it admits the trivial solution $(0,0)$. Since the degree of
the invertible linear operator $L: (x,y) \mapsto (x, 2x + y)$ is $1$, by
Proposition \ref{proposition_proper_plus_bounded} the map
\[
(L+g): (x,y) \mapsto (x - 2\sin(x + x^2 - y^2),2x + y + 1 - \cos(xy))
\]
is proper with degree $1$.

Now, notice that the linearized map of $(L + g)$ at the origin is given by
$(x,y) \mapsto (-x, 2x + y)$, whose determinant is negative. Thus, Theorem
\ref{theorem_nontrivial_solution} implies that the above system has at least
one nontrivial solution (very likely, at least two, because of Sard's Lemma
and the Computation Formula).
\end{example}

\medskip
Degree theory has important applications in the study of bifurcation
problems. Let us see its r\^ole in the finite dimensional context.

Let $J$ be a real interval, $U$ an open subset of $\R^k$ containing the
origin $0 \in \R^k$, and $f: J \times U \to \R^k$ a continuous map
satisfying the condition $f(\lambda,0) = 0$ for any $\lambda \in J$.
Consider the equation
\begin{equation}
\label{equation_bifurcation}
f(\lambda,x) = 0.
\end{equation}
Any pair $(\lambda,0)$, with $\lambda \in J$, is called a \emph{trivial
solution} of the above equation and, consequently, any other solution is said
to be \emph{nontrivial}.

A \emph{bifurcation point} of the equation \eqref{equation_bifurcation} is a
number $\lambda_0 \in J$ (or, equivalently, a trivial solution  $(\lambda_0,0)
\in J \times U$) with the property that any neighborhood of $(\lambda_0,0)$
contains nontrivial solutions.

For example, if $f$ has the special form
\[
f(\lambda,x) = \lambda x - Lx,
\]
where $L$ is a linear operator in $\R^k$, any eigenvalue of $L$ is a
bifurcation point.
Therefore, in some sense, a bifurcation point is the nonlinear analogue of
what in the liner case is the eigenvalue.

Assume now that $f$ is continuously differentiable with respect to the second
variable (at least in a neighborhood of the set $J \times \{0\}$ of trivial
solutions) and let $\lambda_0 \in J$ be given. If the partial derivative
$\partial_2 f(\lambda_0,0)$ is nonsingular, then the Implicit Function
Theorem implies that in a convenient neighborhood $I \times V$ of
$(\lambda_0,0)$ the set $f^{-1}(0)$ is the graph of a map from $I$ to $V$.
Consequently, the assumption $f(\lambda,0) \equiv 0$ implies that
$\lambda_0$ is not a bifurcation point. We have, therefore, the following
result.

\begin{theorem}[Necessary Condition for Bifurcation]
Let $f$ be as above, and consider the (continuous) real function
$\varphi(\lambda) = \det(\partial_2 f(\lambda,0))$, $\lambda \in J$. If
$\lambda_0 \in J$ is a bifurcation point for the equation $f(\lambda,x) = 0$,
then $\varphi(\lambda_0) = 0$.
\end{theorem} 

There are simple examples showing that the condition $\varphi(\lambda_0) = 0$
is not sufficient for $\lambda_0$ to be a bifurcation point. Perhaps, the
simplest one is given by $f: \R^2 \to \R$ defined as $f(\lambda,x) =
(\lambda^2 + x^2)x$, in which $\lambda_0 = 0$.

A sufficient condition for $\lambda_0$ to be a bifurcation point is that
$\varphi(\lambda)$ changes sign at
$\lambda_0$. In fact, we have the following result.

\begin{theorem}[Sufficient Condition for Bifurcation]
\label{theorem_bifurcation_sufficient}
Let $[a,b]$ be a real interval, $U$ an open subset of $\R^k$ containing the
origin $0 \in \R^k$, and $f$ a continuous map from $[a,b] \times U$ into
$\R^k$ satisfying the condition $f(\lambda,0) = 0$ for any $\lambda$ in
$[a,b]$. Assume that $f$ is differentiable with respect to the second
variable at any trivial solution
$(\lambda,0)$ of the equation $f(\lambda,x) = 0$, and define $\varphi: [a,b]
\to \R$ by $\varphi(\lambda) = \det(\partial_2 f(\lambda,0))$.

If $\varphi(a) \varphi(b) < 0$, then the interval $[a,b]$ contains at least
one bifurcation point. In particular, if $\varphi(\lambda)$ has a sign-jump at
some $\lambda_0 \in [a,b]$, then $\lambda_0$ is a bifurcation point.
\end{theorem} 

\begin{proof}
Assume the contrary. Then, by the compactness of $[a,b]$, there exists a
\emph{bounded} open neighborhood $V$ of $0$ such that $\cl V \subseteq
U$ and that the solutions of the equation $f(\lambda,x) = 0$ which are in
$[a,b]
\times
\cl V$ are all trivial. In particular
\[
f(\lambda,x) \neq 0, \quad \forall (\lambda,x) \in [a,b] \times \partial V.
\]
Thus, the Homotopy Invariance Property implies
\[
\deg(f(a,\cdot),V,0) = \deg(f(b,\cdot),V,0).
\]
On the other hand, by the Computation Formula (Theorem
\ref{theorem_computation_formula}), we get
\[
\deg(f(a,\cdot),V,0) = \sign(\varphi(a))\quad \mbox{and} \quad
\deg(f(b,\cdot),V,0) = \sign(\varphi(b)),
\]
contradicting the assumption $\varphi(a) \varphi(b) < 0$.
\end{proof}

The following simple example illustrates how Theorem
\ref{theorem_bifurcation_sufficient} applies.

\begin{example}
The system
\begin{equation}
\label{equation_bifurcation_example}
\left\{
\begin{array}{lcc}
x - \lambda\sin(x + x^2 - y^2) & = & 0\\
2x + y + 1 - \cos xy  & = & 0
\end{array}
\right.
\end{equation}
has a bifurcation point at $\lambda = 1$. To see this consider the linearized
problem
\[
\left\{
\begin{array}{lcc}
x - \lambda x & = & 0\\
2x + y  & = & 0
\end{array}
\right.
\]
(of \eqref{equation_bifurcation_example} at the origin of $\R^2$) and observe
that the function
\[
\varphi(\lambda) = \det
\begin{pmatrix}
1 - \lambda & 0 \\
2 & 1
\end{pmatrix}
\]
has a sign-jump at $\lambda = 1$.
\end{example}

\section{The extended case}\label{excase}
In this section we extend the Brouwer degree to the class of \emph{weakly
admissible} triples; that is, triples of the type $(f,U,y)$, where $U
\subseteq \R^k$ is an open set, $f$ is an $\R^k$\mbox{-}valued map
defined (at least) on $U$, and $y \in \R^k$ is such that $f^{-1}(y) \cap U$
is compact.

The Excision Property of the degree for admissible triples (Theorem
\ref{theorem_excision}) shows that the following definition is well posed.

\begin{definition}
\label{definition_degree_weakly_regular}
The \emph{Brouwer degree} of a weakly admissible triple $(f,U,y)$ is the
integer
\[
\deg(f,U,y)\; := \deg(f,V,y),
\]
where $V$ is any bounded open neighborhood of $f^{-1}(y) \cap U$ such that
$\cl V \subseteq U$.
\end{definition}

Let, as above, $U$ be an open subset of $\R^k$, $H$ an $\R^k$\mbox{-}valued
map defined (at least) on $U \times [0,1]$, and $\alpha: [0,1] \to \R^k$ a
path. The triple $(H,U,\alpha)$ is said to be a \emph{weakly admissible
homotopy of triples} if the set
\[
\Sigma = \Big\{ (x,\lambda) \in U \times [0,1]: H(x,\lambda) =
\alpha(\lambda) \Big\}
\]
is compact. 

\medskip
The following result is a direct consequence of Theorem
\ref{theorem_fundamental_properties_admissible}.

\begin{theorem}
\label{theorem_fundamental_properties_weakly_admissible}
The Brouwer degree for weakly admissible triples satisfies the following three
\emph{fundamental properties}:
\begin{itemize}
\item[] \emph{(Normalization)} $\deg(I,\R^k,0) = 1$, where $I$
denotes the identity on $\R^k$;
\item[] \emph{(Additivity)} if $(f,U,y)$ is weakly admissible, and $U_1$ and
$U_2$ are two disjoint open subsets of $U$ such that $f^{-1}(y) \cap U
\subseteq U_1
\cup U_2$, then
\[
\deg(f,U,y) = \deg(f,U_1,y) + \deg(f,U_2,y);
\]
\item[] \emph{(Homotopy Invariance)} if $(H,U,\alpha)$ is a weakly
admissible homotopy of triples, then
\[
\deg(H(\cdot,0),U,\alpha(0)) = \deg(H(\cdot,1),U,\alpha(1)).
\]
\end{itemize}
\end{theorem}

Given an open subset $U$ of $\R^k$ and an $\R^k$\mbox{-}valued map $f$
defined (at least) on $U$, the integer $\deg(f,U,y)$ does not
necessarily depend continuously on $y$. For instance, the triple
$(\exp,\R,y)$ is weakly admissible for all $y \in
\R$, but the map $y \mapsto \deg(\exp,\R,y)$ is discontinuous at $y = 0$. 
To avoid this inconvenience, given $U$ and $f$ as above, we weed out
a subset of $\R^k$, called \emph{boundary set of $f$ in $U$}, with the
property that the map $y \mapsto \deg(f,M,y)$ turns out to be well defined and
continuous in the complement of this set. Moreover, when $f$ is proper on
$\cl U$, this set coincides with $f(\partial U)$.

Given $y \in \R^k$, we say that $f$ is {\it $y$\mbox{-}proper in $U$} if there
exists a neighborhood $V$ of $y$ such that $f^{-1}(K) \cap U$ is compact for
any compact subset $K$ of $V$ (this means that the restriction of $f$ from $U
\cap f^{-1}(V)$ into $V$ is proper). Clearly, the set
\[
\big\{y \in \R^k: f \mbox{ is $y$-proper in } U \big\}
\]
is open in $\R^k$. Consequently, its complement, called the {\it boundary
set of $f$ in $U$} and denoted by $\partial (f,U)$, is closed.

Clearly $\deg(f,U,y)$ is defined for any $y \in \R^k \setminus
\partial(f,U)$ and, because of the homotopy property, depends continuously on
$y$.

\begin{exercise}
Let $f: U \to \R^k$ be a map defined on an open subset $U$ of $\R^k$. Prove
that $f$ is proper if and only if $\partial (f,U) = \emptyset$.
\end{exercise}

\begin{exercise}
Let $f$ be an $\R^k$\mbox{-}valued map defined (at least) on an open set $U
\subseteq \R^k$. Show that $f(\partial U) \subseteq \partial(f,U)$. If, in
addition, $f$ is proper on $\cl U$, prove that $f(\partial U) =
\partial(f,U)$.
\end{exercise}

The following result is an useful extension of the above Homotopy Invariance
Property.

\begin{theorem}[General Homotopy Invariance Property]
\label{theorem_general_homotopy_invariance} Let $H$ be an
$\R^k$\mbox{-}valued map defined on an open subset $W$ of $\R^k
\times [0,1]$ and $\alpha: [0,1] \to \R^k$ a path. If the set
\[
\Sigma = \big\{ (x,\lambda) \in W: H(x,\lambda) =
\alpha(\lambda) \big\}
\]
is compact, then
\[
\deg(H_\lambda, W_\lambda, \alpha(\lambda))
\]
does not depend on $\lambda \in [0,1]$, where $H_\lambda: W_\lambda \to \R^k$
denotes the partial map $H(\cdot,\lambda)$ defined on the slice $W_\lambda =
\{x \in \R^k: (x,\lambda) \in W \}$.
\end{theorem}

\begin{proof}
Clearly, given $\lambda \in [0,1]$, $\deg(H_\lambda, W_\lambda,
\alpha(\lambda))$ is well defined, since the set 
$H_\lambda^{-1}(\alpha(\lambda))$ coincides with the $\lambda$\mbox{-}slice
$\Sigma_\lambda$ of $\Sigma$, which is compact. Therefore, it is enough to
show that the function $\varphi: [0,1] \to \Z$ given by $\varphi(\lambda)
= \deg(H_\lambda, W_\lambda, \alpha(\lambda))$ is locally constant. For this
purpose, fix any $\mu \in [0,1]$ and consider any bounded open neighborhood
$V$ of $\Sigma_\mu$ such that $\cl V \subseteq W_\mu$. Since $\cl
V$ is compact, there exists a closed neighborhood $I_\delta = [\mu - \delta,
\mu + \delta] \cap [0,1]$ of $\mu$ in $[0,1]$ such that $\cl V \times
I_\delta \subseteq W$.

We claim that, if $\delta$ is sufficiently small, then $\Sigma_\lambda
\subseteq V$ for all $\lambda \in I_\delta$. Assume the contrary.  Thus,
there exists a sequence $\{(x_n,\lambda_n)\}$ in $\Sigma$ such that $\lambda_n
\to \mu$ and $x_n \notin V$ for all $n \in \N$. Because of the compactness of
$\Sigma$ (and the fact that $V$ is open) we may assume that $x_n \to y \notin
V$. Therefore $(y,\mu) \in \Sigma$, which implies $y \in \Sigma_\mu$; and this
is a contradiction since, by assumption, $\Sigma_\mu \subseteq V$.

Assume, without loss of generality,  $\Sigma_\lambda \subseteq V$ for all
$\lambda \in I_\delta$. Since, in addition, $\cl V$ is compact and
contained in $W_\lambda$ for all $\lambda \in I_\delta$, by Definition
\ref{definition_degree_weakly_regular} we get
\[
\deg(H_\lambda, W_\lambda, \alpha(\lambda)) = 
\deg(H_\lambda, V, \alpha(\lambda)), \quad \forall \lambda \in I_\delta.
\]
Now, observe that $H(x,\lambda) \neq \alpha(\lambda)$ for all
$(x,\lambda) \in \partial V \times I_\delta$. Thus, the Homotopy Invariance
Property of the degree (in Theorem
\ref{theorem_fundamental_properties_admissible}) implies that
$\varphi(\lambda)$ does not depend on $\lambda \in I_\delta$, and the 
assertion follows since $\mu$ is arbitrary.
\end{proof}

\section{Appendix: proof of the Homotopy Invariance Property for regular triples}%
\label{appendix}
Our purpose, here, is to prove a crucial result in the construction of the
Brouwer degree: the Homotopy Invariance Property for regular triples (see
Theorem \ref{theorem_fundamental_properties_regular}).

\medskip
Let $U$ be open in $\R^k$ and $f: \cl U \to \R^k$ a proper smooth map
(we recall that $f$ is smooth on $\cl U$ if it admits a smooth
extension on an open set containing $\cl U$). Consider the closed set
\[
K = \{x \in \cl U: \det(f'(x)) = 0 \}
\]
of the critical points of $f$ in $\cl U$. Recalling that proper maps
are closed (see Exercise \ref{exercise_proper_is_closed}), the set
\[
W = \R^k \setminus (f(\partial U) \cup f(K))
\]
of the points $y$ for which $(f,U,y)$ is a regular triple is open. As already
pointed out, for any $y \in W$, the set $f^{-1}(y)$ is compact and
discrete, therefore finite.

\medskip
We need the following lemma which cannot be considered as a particular case of 
Theorem \ref{theorem_continuous_dependence} since the latter has been deduced from 
the three Fundamental Properties.

\begin{lemma}
\label{lemma_locally_constant}
Let $U$, $f$ and $W$ be as above. Then the map
\[
\deg(f,U,\cdot): W \to \Z,
\]
is locally constant.
\end{lemma}

\begin{proof}
Fix any $y \in W$ and let $f^{-1}(y) = \{x_1,x_2, \cdots, x_n\}$. Because of
the Inverse Function Theorem, there exist $n$ pairwise disjoint
neighborhoods $U_1, U_2, \ldots, U_n$ of $x_1, x_2, \ldots, x_n$
which are mapped diffeomorphically onto neighborhoods $V_1$, $V_2$, \ldots,
$V_n$ of $y$. We may assume that each $V_i$ is contained in $W$ and that
in each $U_i$ the sign of $\det(f'(x))$ is constant. Therefore, if $z \in V :=
V_1 \cap V_2 \cap \cdots \cap V_n$ and $\Omega := U_1 \cup U_2 \cup \cdots
\cup U_n$, the equation $f(x) = z$ has exactly $n$ solutions in $\Omega$ and 
\[
\deg(f,U,y) = \deg(f,\Omega,z).
\]
Observe now that, when $z \in \R^k$ is sufficiently close to $y$, the equation
$f(x) = z$ has no solutions in $C = \cl U \setminus \Omega$. Indeed,
this  happens if $z$ does not belong to $f(C)$, which is a closed subset of
$\R^k$ not containing $y$. Thus, if $z$ belongs to the neighborhood $V
\setminus f(C)$ of $y$, we obtain
\[
\deg(f,U,z) = \deg(f,\Omega,z) = \deg(f,U,y),
\]
and the assertion is proved.
\end{proof}

Before proving the Homotopy Invariance Property for regular triples we recall
some important facts regarding the family of (ordered) bases in a finite
dimensional vector space.

Let $\Sigma_k = (e_1, e_2, \cdots, e_k)$ denote the standard basis
of $\R^k$; that is,
\[
e_1 = (1, 0, 0, \cdots, 0),\;  e_2 = (0, 1, 0, \cdots, 0),\;.\; .\; .\;,\;
e_k = (0, \cdots, 0, 1).
\]
A basis $B$ of $\R^k$ is said to be {\em
positively oriented} (in $\R^k$) if it is {\em equivalent} to $\Sigma_k$;
meaning that the transition matrix from $\Sigma_k$ to $B$ has positive
determinant. If this is not the case, $B$ is {\em negatively oriented}. The
{\em status} of a basis $B$ of being positively or negatively oriented is
obviously stable, since the transition matrix depends continuously on $B$ in
the topology of $(\R^k)^k$. Moreover, replacing just one vector of a basis $B$
with its opposite makes $B$ pass from one status to the other one (from
positively oriented to negatively oriented or vice-versa). Finally, we point
out that an ordered basis $B$ of $\R^k$ is positively oriented if and only if
so is the basis $(B, e_{k+1})$ of $\R^{k+1}$ obtained by adding to $B$
(regarded as a basis of the subspace $\R^k \times \{0\}$ of $\R^{k+1}$) the
last vector of  the standard basis $\Sigma_{k+1}$ of $\R^{k+1}$.

Given a $k$\mbox{-}dimensional vector space $E$ and a linear isomorphism
$L: E \to \R^k$, by $L^{-1}(\Sigma_k)$ we mean the preimage under $L$ of the
standard basis $\Sigma_k$. This, of course, is a basis of $E$. It is important
to observe that, if $L$ is an automorphism of
$\R^k$, $\det(L) > 0$ if and only if $L^{-1}(\Sigma_k)$ is a positively
oriented basis of $\R^k$. This elementary fact turns out to be crucial in the
following proof.

\begin{proof}[Proof of the Homotopy Invariance Property for regular triples]
Recall first that a triple $(f,U,y)$ is regular if and only
if so is $(f-y,U,0)$, and in this case
\[
\deg(f,U,y) = \deg(f-y,U,0).
\]
Therefore, putting $G(x,\lambda) = H(x,\lambda) - \alpha(\lambda)$, it is
enough to show that the degree of the two regular triples $(G_0,U,0)$
and $(G_1,U,0)$ is the same, where, as usual, $G_\lambda$ denotes the partial
map $G(\cdot,\lambda)$.

Apply Lemma \ref{lemma_locally_constant} to find an open neighborhood $V$ of
$0 \in \R^k$ made up of regular values for both $G_0$ and $G_1$ and such
that
\[
\deg(G_0,U,z) = \deg(G_0,U,0) \quad \mbox{and} \quad
\deg(G_1,U,z) = \deg(G_1,U,0),
\]
for all $z \in V$. Because of Sard's Lemma, there exists a regular value $y
\in V$ for $G$ in $U \times [0,1]$, and not only for the restriction of $G$
to the boundary (in the sense of manifolds) 
\[
\delta(U\times[0,1])=(U \times \{0\})\cup (U \times \{1\}). 
\]
The assertion now follows if we show that
\[
\deg(G_0 - y,U,0) = \deg(G_1 - y,U,0).
\]
Therefore, we are reduced to proving that {\em if $(F,U,0)$ is a smooth
admissible homotopy of triples, and $0$ is a regular value for $F$ and for
the partial maps $F_0$ and $F_1$, then $\deg(F_0,U,0) = \deg(F_1,U,0)$.}

\smallskip
Assume that $F$ is such a homotopy. Since $0$ is a regular value both for $F$
and the restriction of $F$ to $\delta(U \times [0,1])$, the Regularity Theorem 
(see e.g.\ \cite{GuPo,Hir,Mil}) for manifolds with boundary ensures that  
$F^{-1}(0)$ is a compact $1$\mbox{-}dimensional manifold whose boundary is given 
by
\[
\delta F^{-1}(0) = F^{-1}(0) \cap \delta (U \times [0,1]).
\]
The points of $\delta F^{-1}(0)$ can be divided in two classes:
$A_0 = F_0^{-1}(0) \times \{0\}$ and $A_1 = F_1^{-1}(0) \times \{1\}$,
both finite since $0 \in \R^k$ is a regular value for the partial maps $F_0$
and $F_1$. 

Any point in $\delta F^{-1}(0)$ can be given a
sign $+1$ or $-1$ as follows: if $p = (x,\lambda) \in \delta F^{-1}(0)$, we
put $\sign(p) = \sign(\det(F_\lambda'(x)))$. Thus, we need to prove that
\[
\sum_{p \in A_0} \sign(p) = \sum_{p \in A_1} \sign(p).
\]
This will be done by showing that any point $p \in \delta F^{-1}(0)$ has a
unique companion $c(p) \in \delta F^{-1}(0)$ with the property that
$\sign(p) = -\sign(c(p))$ if and only if both $p$ and $c(p)$ belong to the
same side ($A_0$ or $A_1$).

Recall that any smooth, compact, connected $1$\mbox{-}dimensional real
manifold with nonempty boundary (called an {\em arc}) is diffeomorphic to the
interval $[0,1]$.%
\footnote{This is a consequence of a well-known classification theorem for 
smooth $1$\mbox{-}dimensional real manifolds with (possibly empty) boundary. 
See e.g., \cite{GuPo,Mil}.}
Therefore, any $p \in \delta F^{-1}(0)$ is an endpoint of
an arc (the connected component of $\delta F^{-1}(0)$ containing $p$)
having the other endpoint
$c(p)$ still in $\delta F^{-1}(0)$. Incidentally, observe that the self-map
$c$ of $\delta F^{-1}(0)$ is a bijection (in fact, $c^{-1} = c$).

Consider, for example, the case when the endpoints $p_0$ and $p_1$ of an arc
$M$ contained in $F^{-1}(0)$ are both in $A_0$. We need to show that these two
points have opposite sign. The other two cases (both the endpoints in $A_1$,
or one in $A_0$ and the other in $A_1$) can be treated in a similar way, and
their discussion will be omitted.

Roughly speaking, in order to prove that the two endpoints $p_0 = (x_0,0)$ and
$p_1 = (x_1,0)$ of $M$ have opposite sign we move, continuously, a basis $B_t$
of $\R^{k+1}$ along $M$ in such a way that at the departure (for $t=0$) the
basis coincides with
\[
(F_0'(x_0)^{-1}(\Sigma_k),e_{k+1})
\]
and at the arrival (for $t=1$) coincides with
\[
(F_0'(x_1)^{-1}(\Sigma_k),-e_{k+1}),
\]
where, we recall, $F_0$ stands for the partial map $F(\cdot,0)$.
Since $B_t$ is a basis for all $t \in  [0,1]$, the determinant of the
transition matrix from $B_t$ to $\Sigma_{k+1}$ has constant sign. Thus,  the
two bases $B_0$ and $B_1$ turn out to be either both positively oriented or
both negatively oriented. As a consequence of this, if, for example, the basis
$F_0'(x_0)^{-1}(\Sigma_k)$ of $\R^k$ is positively oriented, the other basis
$F_0'(x_1)^{-1}(\Sigma_k)$ must be negatively oriented, and this implies
\[
\sign(\det(F_0'(x_0))) = -\sign(\det(F_0'(x_1))),
\]
showing that the two endpoints of the arc $M$ have opposite sign.

Let $\gamma(t) = (x(t),\lambda(t))$, $t \in [0,1]$, be a parametrization of
the arc $M$. In other words, $\gamma: [0,1] \to M$ is a diffeomorphism
from $[0,1]$ onto $M$, that we may assume to be oriented from $p_0$ to $p_1$
(i.e.\ $\gamma(0) = p_0$ and $\gamma(1) = p_1$). 

Since $F(\gamma(t)) \equiv 0$, we have $F'(\gamma(t)) \gamma'(t) \equiv 0$.
Moreover, the assumption that $\gamma$ is a diffeomorphism implies
$\gamma'(t) \neq 0$ for all $t \in [0,1]$. Therefore, given $t$, the kernel of
the surjective operator $F'(\gamma(t)): \R^{k+1} \to \R^k$, which is
1\hbox{-}dimensional, is spanned by $\gamma'(t)$. This implies that the restriction 
of $F'(\gamma(t))$ to any $k$\mbox{-}dimensional subspace $E$ of $\R^{k+1}$
not containing $\gamma'(t)$ is an isomorphism, and, consequently, the
preimage of the standard basis $\Sigma_k$ of $\R^k$ under this isomorphism is
a basis of $E$.

To simplify the notation, given a point $p$ in $U \times [0,1]$ and a subspace
$E$ of $\R^{k+1}$, if the restriction $F'(p)|_E$ of $F'(p)$ to $E$ is an
isomorphism, the preimage $(F'(p)|_E)^{-1}(\Sigma_k)$ of the standard
basis $\Sigma_k$ will be denoted by $\Sigma(p,E)$.

Since $\gamma'(t) = (x'(t),\lambda'(t))$ is a nonzero vector for any $t \in
[0,1]$ and the partial derivatives $\partial_1 F(x(0),\lambda(0))$ and
$\partial_1 F(x(1),\lambda(1))$ are invertible (recall that $\lambda(0) =
\lambda(1) = 0$, and $0$ is a regular value for the partial map
$F(\cdot,0)$\,), the identity
\[
\partial_1 F(x(t),\lambda(t))x'(t) + \lambda'(t) \partial_2
F(x(t),\lambda(t)) \equiv 0
\]
yields $\lambda'(0) \neq 0$ and $\lambda'(1) \neq 0$. The fact that
$\lambda(t) \in [0,1]$ for all $t \in [0,1]$ actually implies
$\lambda'(0) > 0$ and $\lambda'(1) < 0$. In other words, denoting by $\langle
\cdot, \cdot \rangle$ the inner product in
$\R^{k+1}$, we have   $\langle \gamma'(0), e_{k+1} \rangle > 0$ and $\langle
\gamma'(1), e_{k+1} \rangle < 0$. 

Define a point $p_t$ moving along the arc $M$ by 
\[
p_t =
\begin{cases}
p_0 & \text{if $t \in [0,1/3]$}\\
\gamma(3t - 1) &  \text{if $t \in [1/3,2/3]$}\\
p_1 & \text{if $t \in [2/3,1]$}
\end{cases}
\]

For any $t \in [0,1]$, define the vector $v_t \in \R^{k+1}$ by 
\[
v_t =
\begin{cases}
(1-3t)e_{k+1}+3t\gamma'(0) & \text{if $t \in [0,1/3]$}\\
\gamma'(3t - 1) &  \text{if $t \in [1/3,2/3]$}\\
(3-3t)\gamma'(1)-(3t-2)e_{k+1} & \text{if $t \in [2/3,1]$}
\end{cases}
\]

Observe that $v_t \neq 0$ for any $t \in [0,1]$. Thus the orthogonal space
$v_t^\perp$ to $v_t$ is always $k$\mbox{-}dimensional. Let us prove that the
restriction of the derivative $F'(p_t)$ to $v_t^\perp$ is an isomorphism for
all $t \in [0,1]$. For this purpose, given $t \in [0,1]$, we need to show
that $v_t^\perp$ does not contain the one dimensional kernel of $F'(p_t)$.
Since $0 \in \R^k$ is a regular value for $F$, this kernel coincides with the
tangent space to $M$ at $p_t$, which is spanned by $\gamma'(0)$ if $t \in
[0,1/3]$, by $\gamma'(3t - 1)$ if $t \in [1/3,2/3]$ and by $\gamma'(1)$ if
$t \in [2/3,1]$.

\smallskip
Consider first the case of $t \in [0,1/3]$. We need to show that
$\gamma'(0)$ does not belong to $v_t^\perp$, which means $\langle v_t,
\gamma'(0) \rangle \neq 0$. In fact, since $\lambda'(0) > 0$, we have
\[
\langle v_t, \gamma'(0) \rangle =
(1-3t)\lambda'(0) + 3t \|\gamma'(0)\|^2 > 0. 
\]

If $t \in [1/3,2/3]$,
\[
\langle v_t, \gamma'(3t-1) \rangle =
\|\gamma'(3t-1)\|^2 > 0.
\]

Finally, it $t \in [2/3,1]$, $\lambda'(1)$ being negative, one gets
\[
\langle v_t, \gamma'(1) \rangle =
(3-3t)\|\gamma'(1)\|^2 - (3t-2)\lambda'(1) > 0. 
\]

Let $t \in [0,1]$. Since, as claimed, the restriction of $F'(p_t)$ to
$v_t^\perp$ is an isomorphism, it makes sense to define the following basis
of $\R^{k+1}$:
\[
B_t = (\Sigma(p_t, v_t^\perp), v_t).
\]

Clearly $B_t$ depends continuously on $t \in [0,1]$, as a map into
$\left(\R^{k+1}\right)^{k+1}$. Observe also that the spaces $v_0^\perp$ and $v_1^\perp$
coincide with $\R^k \times \{0\}$. Therefore, identifying $\R^k$
with the subspace $\R^k \times \{0\}$ of $\R^{k+1}$, we have
\[
B_0 = (F_0'(x_0)^{-1}(\Sigma_k),e_{k+1}) \quad \mbox{and} \quad
B_1 = (F_0'(x_1)^{-1}(\Sigma_k),-e_{k+1}).
\]
Now, as already pointed out, the fact that $B_t$ is always a basis for
$\R^{k+1}$ implies that $B_0$ and $B_1$ are either both positively oriented
or both negatively oriented. Consequently, $B_0$ and
\[
B_1^- = (F_0'(x_1)^{-1}(\Sigma_k),e_{k+1})
\]
have opposite orientation, which implies that also the two bases
\[
F_0'(x_0)^{-1}(\Sigma_k) \quad \mbox{and} \quad
F_0'(x_1)^{-1}(\Sigma_k)
\]
have opposite orientation. Thus,
\[
\sign(\det(F_0'(x_0))) = -\sign(\det(F_0'(x_1))),
\]
and the two endpoints $p_0$ and $p_1$ of the arc $M$ have opposite sign, as
claimed.
\end{proof}

\end{document}